\documentclass[11pt]{amsart}
\usepackage{amssymb,amsfonts,amscd,amstext,amsmath}
\usepackage{epsfig,psfrag,graphicx}
\usepackage{graphicx,psfrag,mathrsfs} 
\usepackage{anysize}
\usepackage{lineno} 
\newtheorem{theorem}{Theorem}[section]
\newtheorem{lemma}[theorem]{Lemma}
\newtheorem{corollary}[theorem]{Corollary}
\newtheorem{proposition}[theorem]{Proposition}
\newtheorem{example}[theorem]{Example}

\newtheorem*{remark}{Remark}
\theoremstyle{definition}
\newtheorem{definition}[theorem]{Definition}
\numberwithin{equation}{section}
\newcommand{\N}{{\mathbb{N}}}

\newcommand{\R}{{\mathbb{R}}}

\def\1{{{\mathit 1} \!\!\>\!\! I} }
\renewcommand{\phi}{\varphi}

\title[Dynamics of piecewise contractions]{Dynamics of piecewise contractions of the interval}


\medskip
\linenumbers

\begin{document}

\maketitle

\centerline{\scshape Arnaldo Nogueira\footnote{Partially supported by ANR Perturbations and EurDynBraz.}}
\smallskip
{\footnotesize
 \centerline{Institut de Math\'ematiques de Luminy,  Aix-Marseille Universit\'e}
 \centerline {163, avenue de Luminy - Case 907, 13288 Marseille Cedex 9, France}
 \centerline{arnaldo.nogueira@univ-amu.fr}} 

\medskip

\centerline{\scshape Benito Pires \footnote{Partially supported by FAPESP-BRAZIL (2009/02380-0 and \mbox{2008/02841-4) and by DynEurBraz.}}}
\smallskip

{\footnotesize
 \centerline{Departamento de Computa\c c\~ao e Matem\'atica, Universidade de S\~ao Paulo}
   \centerline{Av. Bandeirantes 3900, Monte Alegre, 14040-901, Ribeir\~ao Preto - SP, Brazil}
   \centerline{benito@usp.br}
}

\bigskip


\marginsize{2.5cm}{2.5cm}{1cm}{2cm}

\begin{abstract} 
We study the asymptotical behaviour of iterates of piecewise contractive maps of the interval. 
It is known that Poincar\'e first return maps induced by some Cherry flows on transverse intervals are, up to topological conjugacy, piecewise contractions. 
These maps also appear in discretely controlled dynamical systems, describing the time evolution of
manufacturing process adopting some decision-making policies. An injective map $f:[0,1)\to [0,1)$ is a {\it piecewise contraction of
$n$ intervals}, if there exists a partition of the interval $[0,1)$ into $n$ intervals \mbox{$I_1$,  \ldots, $I_n$}   such that for every $i\in\{1,\ldots,n\}$,  the restriction  $f\vert_{I_i}$ is $\kappa$-Lipschitz for some $\kappa\in (0,1)$. We prove that every  piecewise contraction $f$ of  $n$ intervals has at most $n$ periodic orbits. Moreover, we show that every piecewise contraction is topologically conjugate to a piecewise linear contraction.

 \end{abstract}

\maketitle

\section{Introduction}

\noindent
The main subject of this article is the asymptotical behaviour of iterates of  piecewise contractive maps of the interval.
Let \mbox{$0<\kappa<1$} be a constant, $n\ge 1$  an integer and $$0=x_0<x_1<\ldots< x_{n-1}<x_n=1.$$ Let $I_1,\ldots,I_n$ be $n$ pairwise disjoint intervals such that $[0,1)=\bigcup_{i=1}^n I_i$ and, for every
\mbox{$i\in\{1,\ldots,n\}$}, $x_{i-1}$ and $x_i$ are the endpoints of $I_i$. Let $f:[0,1)\to [0,1)$ be an injective map such that $x_1,\ldots,x_{n-1}$ are jump discontinuities of $f$ and $f\vert_{I_i}$ is $\kappa$-Lipschitz for every $i\in\{1,\ldots,n\}$. Such \mbox{map $f$} is called here a
{\it piecewise contraction of $n$ intervals}. 

 A point $p\in [0,1]$ is an {\it $\omega$-limit point} of $x$ if there is a sequence of positive integers
  $n_1< n_2<\cdots$ such that 
  $\lim_{\ell\to\infty} f^{n_{\ell}}(x)=p$. The collection of all such $\omega$-limit points is the
  {\it $\omega$-limit set of $x$}, denoted by $\omega(x)$. We say that $f$ is {\it asymptotically periodic} if $\omega(x)$ is a periodic orbit of $f$ for every $x\in [0,1)$.

This article is motivated by the work of Br\'emont \cite{JB}, where it is proved that every piecewise contraction of $n\ge 2$ intervals can be arbitrarily approximated by an asymptotically periodic piecewise contraction of $n$ intervals having at most $2(n-1)$ periodic orbits. Such result was obtained under the assumption  that the continuity intervals are semi-open (e.g. $I_i=[x_{i-1},x_i)$). We  prove here that $n$ is the sharpest upper bound for the number of periodic orbits of all piecewise contractions of \mbox{$n$ intervals}.  No assumption is made on the definition of the partition $I_1,I_2,\ldots,I_n$.

The dynamics of piecewise contractions of 2 intervals was studied by
 \mbox{Gambaudo and Tresser \cite{GT}}. Examples of order-preserving piecewise contractions of $2$ intervals having irrational rotation number and no periodic orbit appear in Coutinho \cite{RC} and Veerman \cite{PV}. Concerning piecewise contractions of $n\ge 2$ intervals, Guti{\'e}rrez \cite{G1} proved that first return maps to a transverse interval of some Cherry flows are, up to topological conjugacy, piecewise linear contractions having no periodic orbit. The topological conjugacy can be made smooth in many cases \mbox{(see Guti{\'e}rrez \cite{G2})}. Such examples are not typical: arbitrarily small $C^r$-closing perturbations of them yield periodic orbits (see Guti{\'e}rrez and Pires \cite{GP}).
  
 Our main results are the following.
    
 \begin{theorem}\label{main} Every  piecewise contraction of $n$ intervals $f$ has at most
$n$ periodic orbits. Moreover,
  if $f$ has $n$ periodic orbits, then $f$ is asymptotically periodic.
\end{theorem}

For completeness sake, we include here the next result. Its proof is adapted from  \cite[Lemma 3, p. 314]{G1} and is left
to \mbox{Section 6}.

\begin{theorem}\label{main2} Every piecewise contraction of $n$ intervals is topologically conjugate to
a piecewise linear contraction of $n$ intervals whose slopes in absolute value equal $\frac12$. 
\end{theorem}

The proof of Theorem \ref{main} is much easier in the special case where $I_i=[x_{i-1},x_i)$ and
$f\vert _{I_i}$ is increasing for every \mbox{$i\in\{1,\ldots,n\}$} \mbox{(see the Appendix A)}. In this case, all the periodic orbits are attractive and so easily detected: we count them by counting the attractors defined by them. 

Here we consider the general case where $I_i$ can be any of the intervals $(x_{i-1},x_i)$, $[x_{i-1},x_i)$, $(x_{i-1},x_i]$, $[x_{i-1},x_i]$.  We also allow the retriction $f\vert_{I_i}$ to be decreasing for some $i\in\{1,\ldots,n\}$. The general case comes out to be much more difficult to deal with because a new phenomenon appears: the presence of degenerate periodic orbits which attract no other point beyond those in themselves, thus their basins of attraction have empty interior. Since such orbits cannot be detected through their basins of attraction, our approach is to show that each such orbit rules out an attractive periodic orbit. That is achieved through a combinatorial lemma (Lemma \ref{comb}). Counting attractive periodic orbits in the general case is not so easy as counting them in the piecewise increasing case: such result is only provided in Section 4, by means of \mbox{Theorem \ref{weaker}}.

Many mathematical models of flow control systems have their time evolution given by piecewise contractions. An important class are the ``switched flow models", which describe scheduling of many manufacturing systems, where a large amount of work is processed at a unit time (see Tian and Yu \cite {TY}). In this respect, Chase, Serrano and Ramadge \cite{CSR} considered an example of a  switched server system   whose long-term behavior is periodic. The hybrid systems  introduced by Ramadge \cite{R} (see also \cite{CSR})  to model chemical manufacturing systems motivated Schurmann and Hoffmann \cite{SH}  to consider a class of dynamical systems which they called {\it strange billiards}. The name comes from the fact that the system behaves partially as a standard billiard (see Sinai \cite{S}). 

More generally, let  $\Omega\subset\R^n$ be a compact convex region whose boundary $\partial \Omega$ is a topological \mbox{$n$-sphere}. Let $\mathcal{V}$ be an inward-pointing vector field defined on $\partial\Omega$. Assume that a particle  inside $\Omega$ moves with constant velocity until it reaches the boundary $\partial \Omega$ when the velocity instantaneously changes to that of the vector field  $\mathcal{V}$  at the collision point. The motion of  such particle gives a semi-flow on an appropriate quotient space of the tangent bundle over $\Omega$. This semi-flow is called  {\it strange billiard} or {\it pseudo-billiard}. 

In the applied models considered in \cite{CSR,R,SH}, the compact region $\Omega$ is the 
unit $(d-1)$-simplex
$$
\Delta_{d-1}=\{ (x_1, \ldots , x_{d}) \in \mathbb{R}^{d}: x_1+ \ldots + x_d=1,\,\,\textrm{and}\,\, x_i\ge 0\,\,\textrm{for every}\,\,i\}.
$$
Therefore the boundary $\partial \Delta_{d-1}$ consists of the faces $F_i=\{ (x_1, \ldots , x_d) \in \Delta_{d-1}:  x_i=0\}$, $1\leq i \leq d$. Moreover, in these systems,
 the vector field $\mathcal{V}$ is constant along every face $F_i$. In 
 \cite{SH}, it is studied the metric properties of the Poincar\'e first return map induced by the semi-flow on the faces $\displaystyle \cup_{i=1}^d F_i$, in particular they derive the invariant measure of the map. 

Peters and Parlitz \cite{PP} considered  switched flow systems modeled on another phase space: \mbox{$\{ (x_1,\ldots,x_d) \in \Delta_{d-1}: 0\leq x_i\leq b, \,1\le i  \le d\}$}, where $b>0$ is a parameter given by the system. In the same way, MacPhee, Menshikov, Popov and Volkov \cite{MMPV} studied a switched flow system
whose phase space is an equilateral triangle. In both cases, the Poincar\'e maps are piecewise contractions.

Switched flow systems were also considered by Blank and Bunimovich \cite{BB} who studied general dynamical properties of strange billiards. They study the case where $\Omega$ is a convex polyhedron and the vector field $\mathcal{V}$ is not necessarily constant on each face $F_i$. They call  attention that a similar situation occurs for billiards in a strong magnetic or in the gravitational field, where only the angle with the field matters. They prefer to call these dynamical systems {\it pseudo-billiard}. In physics, pseudo-billiard is the name given to a class of Hamiltonian dynamical systems which was studied earlier by  Eleonsky, Korolev and Kulagin \cite{EKK}.  

Now we describe the class of pseudo-billiards to which Theorem \ref{main} can be applied. 
Let $\Omega\subset\R^2$ be a convex $s$-sided polygon and let its boundary $\partial \Omega$
be endowed with the metric induced by the unit interval $[0,1)$. Let $\mathcal{V}:\partial \Omega\to\R^2$ be a piecewise continuous inward-pointing vector field having $r$ discontinuities. The Poincar\'e first return map $P:\partial\Omega\to\partial\Omega$ induced by the corresponding semi-flow has at most $r$ discontinuities. We may identify $P$ with a piecewise continuous map $f:[0,1)\to [0,1)$ having
$n\le r+1$ discontinuities. Here we assume that $f$ is a piecewise contraction of $n$ intervals.

Bruin and Deane \cite{BD} considered a class of  planar piecewise contractions which they proved to be asymptotically periodic. In their work, they explain that their motivation are eletronic circuits and  argue that the existence of dissipation leads one to consider piecewise contractions.  

Another motivation to study the dynamics of  piecewise contractions of the interval comes from ergodic optimization (e.g. see Jenkinson \cite{J}). Precisely, let $f$ be a piecewise contraction of $n$ intervals and $\phi: [0,1] \rightarrow \mathbb{R}$ be a continuous function.  So we may wonder: what can be said about the possible values of the time averages
$$
\lim_{k\rightarrow \infty} \frac{1}{k} \sum_{i=0}^{k-1} \phi(f^i(x)),
$$
where $x\in [0,1)$? Here we give a partial answer to this question.

Notice that if $f$ is a piecewise contraction of $n$ intervals then $f(x_i)\in \{f(x_i^-),f(x_i^+)\}$ for every $i\in\{1,\ldots,n-1\}$, where
$f(x_i^-)=\lim_{\epsilon\to 0^+} f(x_i-\epsilon)$ and $f(x_i^+)=\lim_{\epsilon\to 0^+} f(x_i+\epsilon)$. 
\mbox{Theorem \ref{main}} states that, no matter how we define $f$ at its jump discontinuities, $f$ has at most $n$ periodic orbits.

Other  worth-mentioning results related to contractive/expansive behavior of first return maps of Cherry flows are  Martens, van Strien, de Melo and Mendes \cite{4autores}, and Mendes \cite{PM}. Within the framework of interval exchange transformations, Nogueira proved that periodic orbits are a typical phenomenon within interval exchanges with flip, which he relates to strange billiards \cite[p. 524]{N}.
Recently Nogueira, Pires and Troubetzkoy \cite[Theorem A, p. 3]{NPT} proved that $n$ is the sharp bound for the number of periodic components of every interval exchange transformation with flip or not having $n$ continuity intervals.

The key steps towards the proof of Theorem \ref{main} are the following. Theorem \ref{sm} describes the geometric structure of stable manifolds of regular periodic orbits of $f$.  Theorem \ref{weaker} provides the optimal upper bound for the number of regular (and thus attractive) periodic orbits \mbox{of $f$}. \mbox{Lemma \ref{xykmod}},  which is obtained using Lemma \ref{comb}, is a stronger version of  Theorem \ref{weaker}. \mbox{Theorem \ref{main}} is an immediate corollary of  Lemma \ref{xykmod}. The proof of Theorem \ref{main2} depends only on Lemma \ref{3itens}. 

\section{Trapping intervals and trapping regions}

Henceforth, let $0=x_0<x_1<\ldots<x_{n-1}<x_n=1$ and let $f:[0,1)\to [0,1)$ be a piecewise contraction of $n$ intervals having discontinuities $x_1,x_2,\ldots,x_{n-1}$
 and continuity intervals $I_1,I_2,\ldots, I_n$.

For a set $G\subset [0,1)$, denote by ${\rm int}{(G)}$ the interior of $G$ and by $\overline{G}$ its closure, with respect to the topology of the line $\R$. The boundary of $G$ is the set $\partial G=\overline{G}\setminus {\rm int}\,{(G)}$. In this way, if $I\subset [0,1)$ is an interval with endpoints at $a<b$ then
  ${\rm int}\,(I)=(a,b)$ and  $\overline{I}=[a,b]$.
We omit double parentheses by setting $f(a,b)=f\big ((a,b)\big)=\{f(x)\mid x\in (a,b)\}$.

Let $f^0$ be the identity map on $[0,1)$ and let
$f^\ell=f\circ f\circ \cdots \circ f$ be the $\ell^{\rm th}$-iterate of $f$. The {\it orbit} of a point $p\in [0,1)$ is the
set $O_f(p)=\{f^\ell(p)\mid \ell\ge 0\}$. The point $p$ is {\it periodic} if there exists a positive integer $k$ such that
$f^k(p)=p$. If $k=\min\, \{\ell\ge 1\mid f^{\ell}(p)=p\}$, then $p$ is called a {\it $k$-periodic point}. An orbit is
{\it periodic} (respectively, $k$-periodic) if its points are periodic (respectively, $k$-periodic).

A periodic point $p$ is called {\it internal} if \mbox{$p\in (0,1)\setminus \{x_1,\ldots,x_{n-1}\}$}, otherwise $p$ is called an {\it external periodic point}. Hence, an external periodic point is either $0$ or a discontinuity  of $f$.   
 A periodic orbit $\gamma=O_f(p)$ is {\it internal} if $\gamma\subset (0,1)\setminus \{x_1,\ldots,x_{n-1}\}$, otherwise $\gamma$ is said to be an {\it external periodic orbit}. In this way, a periodic orbit is internal if it contains only internal periodic points.
 
 Throughout this article, interval means an  interval with non-empty interior.
  
 \begin{definition}[Regular/degenerate periodic point]\label{rdpp} A periodic point $p$ of $f$ is {\it regular} if there exists  an interval $J$ containing $p$ whose iterates $f^\ell(J)$, $\ell\ge 1$, are intervals. 
 A periodic point is {\it degenerate} if it is not regular.
 \end{definition}
 
 \begin{lemma}\label{regreg} A periodic point $p$ of $f$ is regular if and only if every point in its orbit is regular. \end{lemma}
 \begin{proof} Let $p$ be a regular $k$-periodic point. By Definition \ref{rdpp}, there exists an interval $J$ such that
  for every $i\in \{0,\ldots,k-1\}$, the $k$-periodic point $f^i(p)$ is contained in the interval $f^i(J)$. Moreover, $f^{\ell}(f^i(J))$ is an interval for every $\ell\ge 0$. Thus $f^i(p)$ is also regular.
  \end{proof}
 
 By Lemma \ref{regreg}, it makes sense to define regular periodic orbit.
 
 \begin{definition}[Regular/degenerate periodic orbit] An orbit $\gamma=O_f(p)$ is {\it regular} if $p$ is a regular periodic point, otherwise $\gamma$ is said to be ${\it degenerate}$.
 \end{definition}
 
 \begin{proposition}\label{intreg} Every periodic orbit of $f$ that contains no discontinuity is regular.  \end{proposition}
 \begin{proof} Let $\gamma=O_f(p)$ be a $k$-periodic orbit of $f$ containing no discontinuity. Firstly suppose that $\gamma$ is internal, thus $\gamma$ is contained in the interior of the set $A=[0,1)\setminus\bigcup_{\ell=0}^{k-1}f^{-\ell}\big(\{x_1,\ldots,x_{n-1}\}\big)$. Let $\epsilon>0$ be so small that  \mbox{$J:=[p-\epsilon,p+\epsilon]$} is contained \mbox{in $A$}. Thus, for every \mbox{$\ell\in\{0,\ldots,k-1\}$},  there exists $i(\ell)\in\{1,\ldots,n\}$ such that
$f^\ell(J)$ is contained in the continuity interval $I_{i(\ell)}$. Consequently,
 the 
first $k$ iterates
 $f(J)$, \ldots, $f^k(J)$ of $J$ are intervals. Moreover, $f^k(J)$ is an interval centered at $p$ of ratio less
than $\kappa^{k}\epsilon$, where $\kappa\in (0,1)$ is the Lipschitz constant of $f$. Thus, $f^k{(J)}\subset J$.  Therefore,
$$f^{\ell}(J)\subset f^{\ell \,{\rm mod}\, k}(J)\subset I_{i(\ell \,{\rm mod}\, k)} \,\,\textrm{for every}\,\, 
\ell\ge 0.$$ In this way,
$f^{\ell}(J)$ is an interval for every $\ell\ge 0$. Now suppose that $\gamma$ is  external, thus $\gamma=O_f(0)$ and $\gamma\cap \{x_1,\ldots,x_{n-1}\}=\emptyset$. Therefore, there exists $\epsilon>0$ such that $J:=[0,\epsilon]$ is contained in $A$. By proceeding as above, we obtain that $f^{\ell}(J)$ is an interval for every $\ell\ge 0$, thus $\gamma$ is regular. 
 \end{proof}
 
 Besides the internal periodic orbits, there exist external periodic orbits that are regular.
 We will prove later that regular periodic orbits are attractive (and so have basin of attraction with non-empty interior) whereas degenerate periodic orbits may have
 the basin of attraction reduced to the periodic orbit itself.

\begin{definition}[Trapping interval]\label{ti} We say that an interval $J$ containing a $k$-periodic \mbox{point $p$} is a {\it trapping interval} of $p$ if its iterates $f(J),\ldots,f^k(J)$ are intervals and $f^k(J)\subset J$. \end{definition}

Next we prove the existence of  a trapping interval which contains every  trapping interval \mbox{of $p$}.

\begin{lemma}\label{existencemaximal} Let $\{J_{\lambda}:\lambda\in\Lambda\}$ be the family of all trapping intervals of the $k$-periodic point $p$, then $\bigcup_{\lambda\in\Lambda} J_{\lambda}$ is  a trapping interval of $p$.
\end{lemma}
\begin{proof} By Definition \ref{ti}, $p\in \bigcup_{\lambda\in\Lambda} J_{\lambda}$ and
$$
f^{\ell}\Big(\bigcup_{\lambda\in\Lambda} J_{\lambda}\Big)=\bigcup_{\lambda\in\Lambda} f^{\ell}\big(J_{\lambda}\big)
$$
is an interval containing $f^{\ell}(p)$ for all $0\le \ell\le k$. Moreover,
$$
f^{k}\Big(\bigcup_{\lambda\in\Lambda} J_{\lambda}\Big)=\bigcup_{\lambda\in\Lambda} f^{k}\big(J_{\lambda}\big)\subset \bigcup_{\lambda\in\Lambda} J_\lambda.
$$
\end{proof}


\begin{lemma}[Existence of trapping intervals]\label{eti} If $p$ is a regular periodic point of $f$ then $p$ admits a maximal trapping interval $J_p$.   
\end{lemma}
\begin{proof} Let $p$ be a regular $k$-periodic point of $f$. By Definition \ref{rdpp}, there exists an interval $K$ containing $p$ such that the iterates $f^{\ell }(K)$, $\ell=0,1,2,\ldots$ are intervals. Let $J=\bigcup_{\ell\ge 0} f^{\ell k}(K)$, thus $f^{m}(J)=\bigcup_{\ell\ge 0} f^{m+\ell k}(K)$ is an interval for all $m\ge 1$. Moreover, $f^k(J)=\bigcup_{\ell\ge 1}f^{\ell k}(K)\subset J$. This proves that $J$ is a trapping interval of $p$. 
The existence of the maximal trapping interval follows now from Lemma \ref{existencemaximal}.
  \end{proof}
  
  \begin{definition} We denote by $J_p$ the maximal trapping interval of a regular periodic \mbox{point $p$}.
\end{definition}
      
  \begin{definition}[Maximal trapping region]\label{maxtr}
Let $\gamma$ be a regular periodic orbit. We call the set $\Omega(\gamma)=\bigcup_{p\in\gamma}J_{p}$ the {\it maximal trapping region} of $\gamma$.
\end{definition}

\begin{proposition}[Trapping region structure]\label{trs} Let $\gamma$ be a regular periodic orbit, then its maximal trapping region $\Omega(\gamma)$ has the following properties:
\begin{itemize}
\item [(TR1)] $f(\Omega(\gamma))\subset\Omega(\gamma)$;
\item [(TR2)] $\gamma=\bigcap_{\ell=0}^\infty f^\ell\big( \Omega(\gamma)\big)$;
\item [(TR3)] $\Omega(\gamma)$ is the union of  $k$ disjoint intervals,
where $k$ is the period of $\gamma$.
\end{itemize}
\end{proposition}
\begin{proof} We have that $f^{\ell}(f(J_p))$ is an interval for all $\ell\ge 0$. Moreover,
$$f^k(f(J_p))= f(f^k(J_p))\subset f(J_p),$$
thus $f(J_p)$ is a trapping interval of $f(p)$, so $f(J_p)\subset J_{f(p)}$. Therefore,
$$
f(\Omega(\gamma))=f\Big(\bigcup_{p\in\gamma} J_{p} \Big)=\bigcup_{p\in\gamma}f\big(J_{p}\big)\subset \Omega(\gamma),
$$
which proves (TR1).

Let $p\in\gamma$, thus $p\in \bigcap_{\ell\ge 0} T^{\ell k}(J_p)$ and
$$
\vert T^{\ell k}(J_p)\vert\le \kappa^{\ell k}\vert J_p\vert,
$$
where $\vert \cdot\vert$ stands for the length of the interval.
Hence, $\bigcap_{\ell\ge 0}T^{\ell k}(J_p)=\{p\}$,
which proves (TR2).

The item (TR3) follows straightforwardly from the Definition \ref{maxtr}. 

\end{proof}


   
 \begin{figure}[htbp]\label{fgh}
\includegraphics[width=16cm]{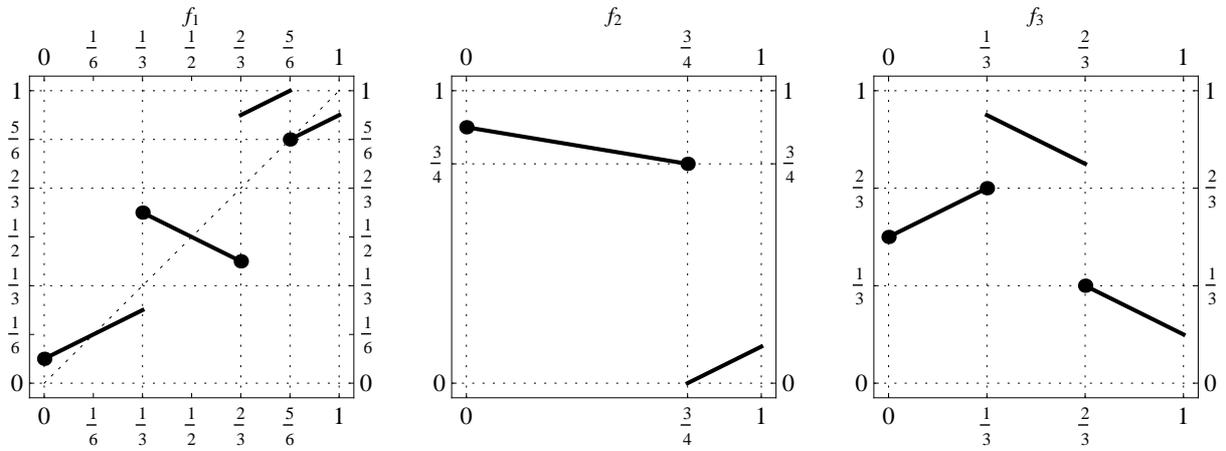}
\caption{Distinct types of periodic points}
\label{default}
\end{figure}

\noindent {\bf Example 1.}
Figure 1 shows the graphs of three piecewise contractions $f_1$, $f_2$ and $f_3$. The points $p_1=\frac16$, $p_2=\frac12$ and $p_3=\frac56$ are regular periodic points of $f_1$. Their maximal trapping intervals are, respectively, $J_{p_1}=[0,1/3)$, $J_{p_2}=[1/3,2/3]$ and
$J_{p_3}=[5/6,1)$. The existence of such trapping intervals are ensured by Lemma \ref{eti}. 

The map $f_2$ shows that the claim of Lemma \ref{eti} is false for the degenerate periodic point $p_4=\frac34$. More precisely, the point $p_4$ is a degenerate external periodic point of $f_2$ that attracts no other point (there is another periodic point that attracts all points of $[0,1)\setminus\{p_4\}$). 

The point $p_5=1/3$ is an external $2$-periodic point of $f_3$ that is also degenerate.\\

\begin{remark}\label{rrr}
The following example shows that it may happen that $\overline{\Omega(\gamma)}\cap\{x_0,\ldots,x_n\}$ is a one-point-set for some regular periodic orbit $\gamma$.
\end{remark}

\begin{figure}[htbp]\label{f9}
\begin{center}
\includegraphics[width=5.5cm]{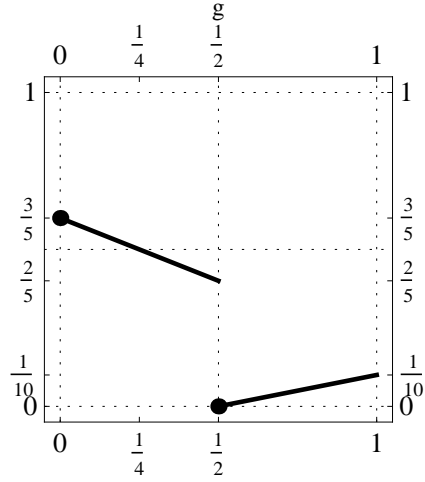}
\caption{Boundary of trapping regions}
\label{default}
\end{center}
\end{figure}

\noindent {\bf Example 2.}
Figure 2 shows the graphs of a 2-interval piecewise contraction $g:[0,1)\to [0,1)$ defined by $g(x)=-0.4x+0.6$ if $x\in [0,0.5)$, otherwise $g(x)=0.2\,(x-0.5)$. The point $p_1=\frac{3}{7}$ is a 1-periodic point of $g$ whereas $p_2=\frac{16}{27}$ is a $2$-periodic point of $g$. Moreover, $J_{p_1}=(\frac{1}{4},\frac{1}{2})$ is the maximal trapping interval of $p_1$ and $J_{p_2}=[\frac{1}{2},1)$ is the maximal trapping interval of $p_2$. For $\gamma=O_g(p_1)$ we have that
$\overline{\Omega(\gamma)}=\overline{J_{p_1}}=[\frac{2}{5},\frac{1}{2}]$. Thus
$\overline{\Omega(\gamma)}\cap \{x_0,x_1,x_2\}=\{x_1\}$, where $x_0=0$,
$x_1=\frac{1}{2}$ and $x_2=1$.\\

\begin{lemma}\label{disj} If $\gamma_1$ and $\gamma_2$ are two distinct regular  periodic orbits of $f$ then
$\Omega(\gamma_1)\cap \Omega(\gamma_2)=\emptyset$.
\end{lemma}
\begin{proof} It follows easily from Proposition \ref{trs}.
\end{proof}
  
  \section{Stable manifolds of periodic orbits}
  
   The {\it stable manifold} (also called the {\it basin of attraction}) of a periodic orbit $\gamma$ of $f$ is the set
  $$
  W^s(\gamma)=\left\{x\in [0,1) \mid \omega(x)=\gamma \right\},
 \quad \textrm{where}\quad
  \omega(x)=\bigcap_{m\in\N}\overline{\{f^{\ell}(x)\mid \ell\ge m\}}.
  $$

  The following lemmas are immediate.
  
  \begin{lemma}\label{fww}  Let $\gamma$ be a periodic orbit, then
  $f(W^s(\gamma))\subset W^s(\gamma)$.
  \end{lemma}
  \begin{proof} Let $x\in W^s(\gamma)$. Then $\omega(x)=\gamma$, that is,
  ${\bigcap_{m\in\N} \overline{\{f^{\ell}(x)\mid \ell\ge m\}}}=\gamma$. Hence,
  $$\omega(f(x))={\bigcap_{m\in\N}\overline{ \{f^{\ell+1}(x))\mid \ell\ge m\}}}={\bigcap_{m\in\N} \overline{\{f^{\ell}(x)\mid \ell\ge m\}}}=\gamma,$$
  and so $f(x)\in W^s(\gamma)$.
  \end{proof}
  
  \begin{lemma}\label{ime} If $\gamma_1$ and $\gamma_2$ are two distinct regular periodic orbits of $f$ then
  $W^s(\gamma_1)\cap W^s(\gamma_2)=\emptyset$.
  \end{lemma}

The stable manifold of a regular periodic orbit $\gamma$ contains the trapping region of $\gamma$, that is, $\Omega(\gamma)\subset W^s(\gamma)$. The stable manifold of a periodic orbit may also include finite sets or intervals that are attracted by the trapping region. \\

\begin{figure}[htbp]\label{f9}
\begin{center}
\includegraphics[width=11.5cm]{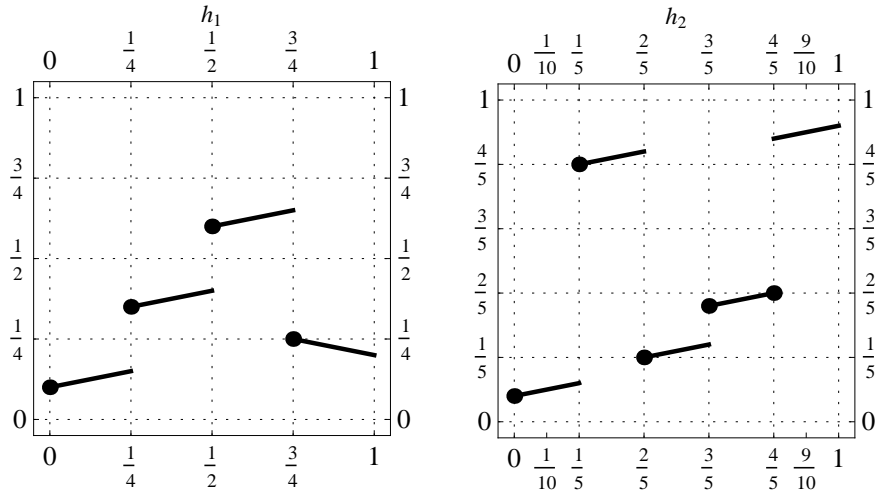}
\caption{Stable manifolds of periodic orbits}
\label{default}
\end{center}
\end{figure}
    
  \noindent {\bf Example 3.} 
 In Figure 3, the map $h_1:[0,1)\to [0,1)$ is a piecewise contraction of 4 intervals. The point $p=\frac{3}{8}$ is a fixed point of $h_1$. Besides, the 1-periodic orbit $\gamma=O_{h_1}(p)$ is internal and so  regular. It is easy to show that $W^s(\gamma)=[\frac14,\frac12)\cup \{\frac34\}$. 
 
In Figure 3, the map $h_2:[0,1)\to [0,1)$ is a piecewise contraction of 5 intervals having positive constant slope. Notice that the 1-periodic points $p_1=\frac{1}{10}$ and $p_2=\frac{9}{10}$ are regular
whereas the $3$-periodic point $p_3=\frac15$ is degenerate. Moreover, the stable manifolds of $\gamma_1=O_{h_2}(p_1)$ and $\gamma_2=O_{h_2}(p_2)$ satisfy $W^s(\gamma_1)\cup W^s(\gamma_2)=[0,1)\setminus \gamma_3$.
 
 In Figure 1, $\gamma=\{\frac34\}$ is a degenerate 1-periodic orbit of $f_2$ such that $W^s(\gamma)=\{\frac34\}$. \\
 
 In general, the geometric structure of a stable manifold of a regular periodic orbit is given by the next result, which turns out to be of paramount importance for the proof of Theorem \ref{main}.
    
\begin{theorem}\label{sm} If $\gamma$ is a regular periodic orbit of $f$ then the interior of $W^s(\gamma)$ is the union of finitely many open intervals.
\end{theorem} 

We postpone the proof of Theorem \ref{sm} to the end of this section. Now we will describe the key points necessary for its proof.

Firstly, we will define a family of finitely many pairwise disjoint open intervals $F_1$, $F_2$, \ldots, $F_r$ whose iterates $f^{\ell}(F_j)$ never meet the discontinuity set
$\{x_1,\ldots,x_{n-1}\}$ of $f$. In this way, $f^{\ell}(F_j)$ is an interval for every $j\in\{1,\ldots,r\}$ and ${\ell}\ge 0$.
The next step is to show that the union of the forward orbits $O_f(F_j)=\bigcup_{{\ell}=0}^\infty f^{\ell}(F_j)$ covers  the interval $[0,1)$ up to a null Lebesgue measure set. In this way, eventually some of these intervals will enter the trapping regions of the regular periodic orbits and stay there thereafter. It follows from Proposition \ref{trs} that
the orbit of an interval $F_j$ can enter at most one trapping region. The time that the interval $F_j$ takes to be captured by a trapping region $\Omega(\gamma)$ of a regular periodic orbit $\gamma$ is called the {\it target time} and is denoted by $\tau(F_j,\gamma)$. We set $\tau(F_j,\gamma)=+\infty$ if $O_f(F_j)\cap \Omega(\gamma)=\emptyset$.

Theorem \ref{sm} then will follow once we prove that for each regular periodic orbit $\gamma$
\begin{equation}\label{fund}
{\rm int}\,(W^s(\gamma))={\rm int}\,(\Omega(\gamma))\cup\bigcup_{j\in\Lambda(\gamma)}\bigcup_{\ell=0}^{{\tau(F_j,\gamma)-1}}f^{\ell}(F_j)\,\,\,\, (\textrm{up to a null Lebesgue measure set}),
\end{equation}
where
$$\Lambda(\gamma)=\{j\in\{1,\ldots,r\}\mid \tau(F_j,\gamma)<+\infty\}.
$$

Hereafter, we will implement the recipe described above in order to prove Theorem \ref{sm}.
  
 Let $E$ be the open set defined by
$$E= {\rm int}\,\Big([0,1)\setminus f\big([0,1)\big)\Big).$$
Notice that $E$ is the union  of at most $n+1$ open intervals $E_1$,$E_2$,\ldots,$E_s$. Moreover, the following is true.

\begin{lemma}\label{toos} For every positive integer $\ell$, 
$E\cap f^{\ell}(E)=\emptyset$.
\end{lemma}
\begin{proof} The assertion follows from the fact that $E\subset [0,1)\setminus f([0,1))$ and $f^{\ell}(E)=f\big(f^{\ell-1}(E)\big)\subset f([0,1))$.
\end{proof}

Now let $B$ be the set consisting of those points of $E$ which are taken by some iterate of $f$ into a discontinuity of $f$, that is:
\begin{equation}\label{BBBB}
B=E\cap \bigcup_{\ell=0}^{+\infty} f^{-\ell}\big(\{ x_1,\ldots,x_{n-1} \}\big).
\end{equation}

 \begin{lemma}\label{Bisfinite} The set $B$ has at most $n-1$ elements.
\end{lemma}
\begin{proof}  We claim that the set
$E\cap\bigcup_{\ell=0}^{+\infty} f^{-\ell}(\{x_j\})$ has at most one element for each $j\in\{1,\ldots,n-1\}$, otherwise the injectivity of $f$ would imply that there exist $x,y\in E$, $x\neq y$, and
$0\le m<\ell$ such that $f^m(x)=x_j=f^\ell(y)$. Hence $x=f^{\ell-m}(y)$. In particular, $E\cap f^{\ell-m}(E)\neq\emptyset$, which contradicts Lemma \ref{toos}.
Therefore, the claim is true and $B$ has at most $n-1$ elements.
\end{proof}

A {\it measurable partition of $[0,1)$ into intervals} is a denumerable family of open, pairwise disjoint intervals $A_1$, $A_2$, $A_3$, $\ldots$ such that $[0,1)\setminus\bigcup_{j=1}^\infty{A_j}$ has Lebesgue measure zero. 


  \begin{lemma}\label{3itens} The set $F=E\setminus B$ is  the union of $r\le 2n$ pairwise disjoint open intervals $F_1$, $F_2$,\ldots, $F_r$. Moreover
     \begin{itemize}
   \item [(a)] $F\cap f^{\ell}(F)=\emptyset$ for every positive integer $\ell\ge 0$;
   \item [(b)] $f^{\ell}(F_j)\subset [0,1)\setminus\{x_0,x_1,\ldots,x_{n-1}\}$ for every $\ell\ge 0$ and $j\in\{1,\ldots, r\}$;
  \item [(c)] $\{f^{\ell}(F_j)\mid \ell\ge 0\,\,\textrm{and}\,\, j\in\{1,\ldots,{r}\}\}$ is a measurable partition of $[0,1)$ into open intervals.
     \end{itemize}
  \end{lemma}
  \begin{proof} It follows from Lemma \ref{Bisfinite} that $F$ is the union of finitely many disjoint intervals \mbox{$F_1$, $\ldots$, $F_r$}.
  By Lemma \ref{toos}, $F\cap f^{\ell}(F)\subset E\cap f^{\ell}(E)=\emptyset$ for every ${\ell}>0$.
 The item (b) follows immediately from the definition of $F$. It follows from (a), (b) and the injectivity of $f$ that the sets $f^{\ell}(F_j)$ form a family of pairwise disjoint intervals. It remains to prove that $[0,1)\setminus \bigcup_{{\ell},j} f^{\ell}(F_j)$ has Lebesgue measure zero. Suppose that this is false and let $(a,b)\subset[0,1)\setminus \bigcup_{{\ell},j} f^{\ell}(F_j)$. Then there exists $0<\lambda<1$ such that for every $m\ge 0$,
  $f^{-m}(a,b)$ is the union of finitely many open intervals of total length not smaller than $\frac{1}{{\lambda}^m}(b-a)$. This is not possible because $f^{-m}(a,b)\subset [0,1)$ for every $m\ge 0$
  and $\frac{1}{{\lambda}^m}(b-a)\to+\infty$ as $m\to+\infty$.
 \end{proof}
 
  \begin{lemma}\label{goesin} Let $\Omega(\gamma)$ be the maximal trapping region of a regular periodic orbit $\gamma$. For each \mbox{$j\in\{1,\ldots,r\}$} and for each $\ell\ge 0$, either $f^{\ell}(F_j)\cap\Omega(\gamma)=\emptyset$ or $f^{\ell}(F_j)\subset \Omega(\gamma)$. 
         \end{lemma}
         \begin{proof} Suppose that $f^{\ell}(F_j)\cap\partial\Omega(\gamma)\neq\emptyset$. By (TR3) of Proposition \ref{trs}, $\partial\Omega(\gamma)$ is a finite point-set. So there exists $\ell'\ge \ell$ such
         that (i) $f^{\ell'}(F_j)\cap \partial\Omega(\gamma)\neq\emptyset$ and (ii) $f^m(F_j)\cap \partial\Omega(\gamma)=\emptyset$ for every $m>\ell'$. By item (b) of Lemma  \ref{3itens}, $\omega(x)=\omega(y)$ for every $x,y\in f^{\ell'}(F_j)$. By (i),  $f^{\ell'}(F_j)\cap{\rm int}\,(\Omega(\gamma))$ has non-empty interior.
         By \mbox{Proposition \ref{trs}}, for every $x\in f^{\ell'}(F_j)\cap{\rm int}\,(\Omega(\gamma))$ we have that $\omega(x)=\gamma$. By the above, $\omega(x)=\gamma$ for every $x\in f^{\ell'}(F_j)$. This together with (ii) imply $f^m(F_j)\subset \Omega(\gamma)$ for every $m>\ell'$. In this way,
         if $U:=f^{\ell}(F_j)\cup \Omega(\gamma)$ then by (TR1) of Proposition \ref{trs} and by the above,
         $$f(U)\subset f^{\ell+1}(F_j)\cup f(\Omega(\gamma))\subset \Omega(\gamma)\subset U.$$
         This contradicts $\Omega(\gamma)$ being a maximal trapping region. So $f^\ell(F_j)\cap\partial\Omega(\gamma)=\emptyset$.
         \end{proof}
         
  \begin{corollary}\label{capturesatleastone} Let $\Omega(\gamma)$ be the maximal trapping region of a regular periodic orbit $\gamma$, then there exist $\ell\ge 0$ and $j\in \{1,\ldots,r\}$ such that $f^\ell(F_j)\subset\Omega(\gamma)$.
  \end{corollary}
  \begin{proof} It follows from item (c) of Lemma \ref{3itens} and from Lemma \ref{goesin}.
  \end{proof}

    \begin{proof}[Proof of Theorem \ref{sm}] For each $j\in\{1,\ldots,r\}$, let $\tau(F_j,\gamma)=\inf\, \{\ell\in\N\mid f^{\ell}(F_j)\subset\Omega(\gamma)\}$, where
$\inf\emptyset=+\infty$. Let $\Lambda(\gamma)=\{j\in\{1,\ldots,r\}\mid \tau(F_j,\gamma)<+\infty\}$. It follows from \mbox{Corollary \ref{capturesatleastone}} that $\Lambda(\gamma)\neq\emptyset$. Now Proposition \ref{trs} and Lemmas \ref{3itens}  and \ref{goesin} ensure that the following statements are true:
\begin{itemize}
\item [(I)] $f^{\ell}(F_j)\cap\Omega(\gamma)=\emptyset$ if $j\in\Lambda(\gamma)$ and $0\le \ell < \tau(F_j,\gamma)$;
\item [(II)] $f^{\ell}(F_j)\subset\Omega(\gamma)$ if $j\in\Lambda(\gamma)$ and $\ell\ge \tau(F_j,\gamma)$;
\item [(III)] $O_f(F_j)\cap\Omega(\gamma)=\emptyset$ if $j\not\in\Lambda(\gamma)$.
\end{itemize} 
Let 
$$S={\rm int}\,(\Omega(\gamma))\cup\bigcup_{j\in\Lambda(\gamma)} \bigcup_{\ell=0}^{\tau(F_j,\gamma)-1} f^{\ell} (F_j).$$ By (TR3) of Proposition \ref{trs} and by item (c) of Lemma \ref{3itens}, $S$ is the union of finitely many open intervals. By item (II) above and (TR2) of \mbox{Proposition \ref{trs}}, we have that $S\subset W^s(\gamma)$. In particular, $S\subset {\rm int}\,(W^s(\gamma))$ because $S$ is open.

It follows from (c) of Lemma \ref{3itens} and from itens (I)-(III) above that $$\Omega(\gamma)=\bigcup_{j\in\Lambda(\gamma)}\bigcup _{\ell\ge\tau(F_j,\gamma)}
f^{\ell}(F_j)\,\, (\textrm{up to a null Lebesgue measure set}).$$
Hence, $S=\bigcup_{j\in\Lambda(\gamma)}\bigcup_{\ell=0}^{+\infty} f^{\ell}(F_j)\,\,(\textrm{up to a null Lebesgue measure set}).$

By (III), 
$$
W^s(\gamma)\setminus S=W^s(\gamma)\Big\backslash\bigcup_{j\in\Lambda(\gamma)}\bigcup _{\ell=0}^{+\infty}
f^{\ell}(F_j)=W^s(\gamma)\Big\backslash\bigcup_{j=1}^r\bigcup _{\ell=0}^{+\infty}
f^{\ell}(F_j)\,\,(\textrm{up to a null measure set}).$$
By (c) of Lemma \ref{3itens}, $\bigcup_{j=1}^r \bigcup_{\ell=0}^{+\infty} f^{\ell}(F_j)$ has Lebesgue measure one and so
$W^s(\gamma)\setminus S$ has Lebesgue measure zero. In particular,  ${\rm int}\,(W^s(\gamma))\setminus S$ has empty interior. 

Suppose that ${\rm int}\,(W^s(\gamma))$ is not an union of finitely many open intervals. Then there exists a denumerable family of pairwise disjoint open intervals $U_1,U_2,\ldots$ such that ${\rm int}\,(W^s(\gamma))=\bigcup_{j=1}^\infty U_j$. Moreover, because $S$ is the union of finitely many pairwise disjoint open intervals and $S\subset {\rm int}\,(W^s(\gamma))$, there exists a positive integer $d$ such that
$S\subset \bigcup_{j=1}^d U_j$. Then $W^s(\gamma)\setminus S$ contains the open set
$U_{d+1}$ which is a contradiction since $W^s(\gamma)\setminus S$ has Lebesgue measure zero.
\end{proof}

 \section{A tight upper bound for the number of regular periodic orbits}

In this section we will present a complete proof of the following result.

\begin{theorem}\label{weaker} Every piecewise contraction of $n$ intervals has at most $n$
regular periodic orbits.
\end{theorem}

By Proposition \ref{intreg}, Theorem \ref{weaker} asserts that a piecewise contraction of $n$ intervals has at most $n$ internal periodic orbits. The proof of the Main Theorem (Theorem \ref{main}) is a variation of the proof of Theorem \ref{weaker}. The steps necessary for obtaining it from the Proof of Theorem \ref{weaker} will be outlined in the next section.

Let \mbox{$\gamma_1$, $\gamma_2$, \ldots, $\gamma_m$} be a collection of pairwise distinct regular periodic orbits of $f$. Set
$W_j={\rm int}\,(W^s(\gamma_j))$ for every $j\in\{1,\ldots,m\}$ and let $W_{m+1}={\rm int}\,([0,1)\setminus\bigcup_{j=1}^m \overline{W_j})$. By Theorem \ref{sm}, $W_{m+1}$ is the union of finitely
many intervals. Moreover, $\bigcup_{j=1}^{m+1} \overline{W_j}=[0,1]$. 

Theorem \ref{weaker} states that $m\le n$. Its proof follows straightforwardly from the next lemmas.

\begin{lemma}\label{arefi} For every $j\in \{1,\ldots,m+1\}$ we have that ${f(W_{j})}\subset \overline{W_{j}}$. 
\end{lemma}
\begin{proof} Firstly, let $j\in\{1,\ldots,m\}$. By Theorem \ref{sm}, $W_j$ is the union of finitely many open intervals.
Therefore, as $f$ is injective and $f\vert_{I_i}$  continuous for every $i\in\{1,\ldots,n\}$,   $f(W_j)$ is the union of finitely many  intervals. 
 By absurd assume that $f(W_j) \cap (\R\setminus \overline{W_j}) \neq \emptyset$, in others words, suppose that  there exists an interval $U \subset f(W_j)$ which intersects the non-empty open set $\R\setminus \overline{W_j}$. 
Therefore the set  $U\cap (\R\setminus \overline{W_j})$ has a non-empty interior. 
 \mbox{Lemma \ref{fww}}  together with the definition  $W_j={\rm int}\,(W^s(\gamma_j))$ yield
 $$
 U \subset f(W_j) \subset f(W^s(\gamma_j))\subset W^s(\gamma_j).
 $$
 On the other hand,  
the set $W^s(\gamma_j)\setminus {W_j}$ has empty interior which contradicts our assumption. 
So $f(W_j)\subset \overline{W_j}$ and the claim holds for every $j\in\{1,\ldots,m\}$.
 
 Now we consider the case $j=m+1$. By absurd assume that $f(W_{m+1})\not\subset  \overline{W_{m+1}}$,
 therefore there exist $x\in W_{m+1}$ and $j,k\in\{1,\ldots,m\}$
 such that $f(x)\in {\rm int}\,(\overline{W_j\cup W_k})$, where  $j=k$ may happen. 
For $\epsilon>0$ small enough 
 $(f(x)-\epsilon,f(x) +\epsilon)\subset {\rm int}\,(\overline{W_j\cup W_k})$ and,  by \mbox{Theorem \ref{sm}}, $(f(x)-\epsilon,f(x))\cup (f(x),f(x) +\epsilon)\subset W_j\cup W_k$.  Let  $i\in\{1,\ldots,n\}$ be the unique index such that $x\in I_i$. 
 As $f$ is injective and $f\vert_{I_i}$ continuous, there exists
 $\delta>0$ such that 
 $$f(x-\delta,x)\subset (f(x)-\epsilon,f(x)) \cup (f(x),f(x) +\epsilon)\,\,\textrm{or}\,\,f(x,x+\delta)\subset (f(x)-\epsilon,f(x))\cup (f(x),f(x) +\epsilon).$$ 
 Therefore, by Lemma \ref{fww}, either $(x-\delta,x)\subset W_j\cup W_k$ or
 $(x,x+\delta)\subset W_j\cup W_k$. This contradicts the fact that $x\in W_{m+1}$.
Thus $f(W_{m+1})\subset \overline{W_{m+1}}$.
\end{proof}
\begin{lemma}\label{inside} If $z\in \overline{W_{i}}\cap \overline{W_{j}}$ for some $i\neq j$ then
there exists an integer $q\ge 0$ such that $f^q(z)\in\{x_1,\ldots,x_{n-1}\}\cap \partial W_i\cap \partial W_j$.
\end{lemma}
\begin{proof} We may assume that $z\not\in \{x_1,\ldots,x_{n-1}\}$, otherwise the proof is finished by taking
$q=0$. Thus, $f$ is continuous in a neighborhood
of $z$. By continuity of $f$ and Lemma \ref{arefi}, we have that $f(z)\in \overline{W_{i}}\cap \overline{W_{j}}$ and the reasoning can be repeated. Hence, we may assume that $f(z)\not\in \{x_1,\ldots,x_{n-1}\}$, otherwise
we set $q=1$ and the proof is finished. By repeating this reasoning over and over again, we obtain that
either $f^q(z)\in \{x_1,\ldots,x_{n-1}\}$ for some $q\ge 0$ (and the proof is finished) or
\mbox{$O_f(z)\cap \{x_1,\ldots,x_{n-1}\}=\emptyset$}. This together with Lemma \ref{arefi} yield $O_f(x)\subset\overline{W_{i}}\cap\overline{W_{j}}$. By Theorem \ref{sm}, $\overline{W_{i}}\cap\overline{W_{j}}$ is a finite point-set, thus $O_f(z)$ is a periodic orbit. By Proposition \ref{intreg}, $O_f(z)$ is regular periodic orbit, which contradicts  $O_f(z)\subset\overline{W_{i}}\cap\overline{W_{j}}$. Hence, there exists an integer $q\ge 0$ such that $f^q(z)\in \{x_1,\ldots,x_{n-1}\}$.
By Lemma \ref{arefi} and by \mbox{Theorem \ref{sm}}, $\overline{W_i}\cap\overline{W_j}\subset
\partial W_i\cap\partial W_j$.
 \end{proof}
 
 \begin{lemma}\label{xyk} The following statements are true:
 \begin{itemize}
 \item [(a)] if $W_{m+1}\neq\emptyset$ then $m\le n-1$;
 \item [(b)] if $W_{m+1}=\emptyset$ then $m\le n$.
 \end{itemize}
 \end{lemma}
 \begin{proof} Firstly, let us prove (a). Let $\mathcal{W}=\{W_1,\ldots,W_{m+1}\}$.
 We will define an injective map 
 \begin{equation}\label{beta}
\beta:\mathcal{W}\rightarrow \{x_{0},\ldots,x_{n-1}\}.
\end{equation} Set $y_j=\inf W_j$ for  all $j\in\{1,\ldots,m+1\}$.
By definition and by Lemma \ref{ime},
\begin{equation}\label{01}
W_1,\ldots,W_{m+1} \; \mbox{are pairwise disjoint and} \;  [0,1]=\bigcup_{j=1}^{m+1} \overline{W_j}.
\end{equation}
Let $j_0\in\{1,\ldots,m+1\}$ be the index that satisfies $y_{j_0}=x_0=0$. Set $\beta(W_{j_0})=x_0$.  
 By (\ref{01}) and by \mbox{Theorem \ref{sm}}, $y_1$,$y_2$,\ldots,$y_{m+1}$ are pairwise disjoint. 
Let $i\in\{1,\ldots,m+1\}$, $i\neq j_0$, thus there exists $W^{(i)} \in \mathcal{W},W^{(i)}\neq W_i$, such that $y_i \in \partial W^{(i)}\cap \partial W_i$. Moreover, for $\epsilon$ small enough, we have
$$
(y_i-\epsilon,y_i) \subset W^{(i)}  \; \mbox{ and } \; (y_i,y_i+\epsilon)\subset W_i.
$$
Using Lemma \ref{inside}, let $q_i=\min\{q\geq 0:f ^q(y_i)\in \{x_1,\ldots,x_{n-1}\}\}$ and set $\beta(W_i)=f^{q_i}(y_i)$. 
 
Now we show that the map $\beta$ is injective. Let $1\leq i, k \leq m+1$, with $q_i\leq q_k$, be such that $\beta(W_i)=\beta(W_k)$. It is easy to see that $i=j_0$ or $k=j_0$ imply $i=k=j_0$. Thus we may assume that $i\neq j_0$ and $k\neq j_0$. By the injectivity of $f$, 
 $$
 f^{q_k-q_i}(y_k)=y_i,
 $$
 where $0\leq q_k-q_i\leq  q_k$. Notice that $q_k=0$ implies $q_i=0$. In this case, $y_i=y_k$, which contradicts $q_1,q_2,\ldots,q_{m+1}$ are pairwise disjoint. Hence, we may assume that $q_k\ge 1$.
 We have that $f$ is continuous on a neighborhood of $f^j(y_k)$ for every $0 \leq j \leq q_k-1$.  Therefore, by \mbox{Lemma \ref{arefi}}, $\{W_i,W^{(i)}\}=\{W_k,W^{(k)}\}$. There are two possibilities: either (i) $W_i=W_k$ or
 $
(ii)\,\,W_i=W^{(k)} \; \mbox{and} \;  W_k=W^{(i)}.
 $
 Suppose that (ii) happens. If $y_i<y_k$ then by $(\ref{01})$, there exists $\epsilon>0$ such that
 $(y_i-\epsilon,y_i)\subset W_k$ and thus $\inf W_k<y_i<y_k$, which is a contradiction.
 By analogy, assuming $y_k<y_i$ also yields a contradiction. Therefore, (ii) cannot happen
 and so $W_i=W_k$. This proves that $\beta$ is a well defined injective map, thus $m \leq n-1$.
 To prove (b), we neglect $W_{m+1}$ and define $\mathcal{W}=\{W_1,\ldots,W_m\}$. Replacing in the above proof $m+1$ by $m$, we obtain that $m-1 \leq n-1$, thus $m\leq n$.
 \end{proof}
 
 Notice that Theorem \ref{weaker} is a corollary of Lemma \ref{xyk}.

\section{Proof of  Theorem \ref{main}}

In this section we will prove Theorem 1.1. In this respect, the combinatorial lemma we present now
is going to be of paramount importance. We will keep the notation of  previous sections.

\subsection{The Combinatorial Lemma}
 
 An {\it $s$-chain} is a collection  of $s\ge 1$ pairs of positive integers
 $$A_0=(a_0,b_0),
 A_1=(a_1,b_1),\ldots,A_{s-1}=(a_{s-1},b_{s-1})$$ satisfying 
$
a_{\ell\, {\rm mod} \, s}=a_{\ell-1}$ or $b_{\ell\,{\rm mod}\,s}=a_{\ell-1},
$
for every $\ell\in \{1,2,\ldots,s\}$.
The set $\displaystyle S=\bigcup_{\ell=0}^{s-1} \{a_{\ell}\}\cup \{b_{\ell}\}$ is its {\it set of coordinates} whose cardinality is denoted by $\# S$.\\

\begin{example}
If $A_{\ell}=(1,\ell+2)$ for every $\ell\in\{0,\ldots,s-1\}$ then
$$
S = \{1,2, \ldots , s+1\}\quad {\rm and}\quad \#S=s+1.
$$
\end{example}
\begin{example}
If $s=4$, $A_0=(1,2),A_1=(1,3),A_2=(4,1)$ and $A_3=(2,4)$ then
$$
S = \{1,2,3,4\} \quad {\rm and}\quad \#S=4=s.
$$
\end{example} 

We would like to know how large the set $S$ can be in the general case. \\

\begin{lemma}[Combinatorial Lemma]\label{comb} If $A_0=(a_0,b_0 )$,  $A_1=(a_1,b_1 )$, \ldots, $A_{s-1}=(a_{s-1},b_{s-1} )$ is an $s$-chain then
$\#S \leq s+1$. Moreover, $\#S= s+1$ if and only if  $a_0=a_1= \ldots = a_{s-1}$ and the elements $a_0,b_0, b_1, \ldots, b_{s-1}$ are pairwise distinct.
\end{lemma}
\begin{proof} The assertion follows by induction on $s$.
The claim holds for $s=1$. Now assume that the claim holds for some $s\geq 1$. Let $A_0=(a_0,b_0),\ldots, A_{s-1}=(a_{s-1},b_{s-1}),A_s=(a_s,b_s)$ be an $(s+1)$-chain and let $S$ be its set of coordinates. We have to prove that $\#S\le s+2$.

If $a_{s-1} =a_0$ or $a_{s-1} =b_0$, then $A_0,A_1,\ldots , A_{s-1}$  is an $s$-chain, then by the induction hypothesis the set 
$\displaystyle \cup_{\ell=0}^{s-1} \{a_{\ell}\}\cup \{b_{\ell}\}$ has at most $s+1$ elements. Now if we add $a_s$ and $b_s$, as at least one of them is also equal to $a_0$ or $b_0$, the set $S$ must have at most $(s+1)+1$ elements.

Otherwise, $a_{s-1} \neq a_0$ and $a_{s-1} \neq b_0$, thus $b_s= a_{s-1}$ and $a_s=a_0$ or $a_s=b_0$ which means that $\displaystyle S=\cup_{\ell=0}^{s-1} \{a_{\ell}\}\cup \{b_{\ell}\}$. One of the coordinates of $A_{s-1}$ equals $a_{s-2}$. Now we replace the couple $(a_{s-1}, b_{s-1} )$ by the couple $(a_0,a_{s-2})$, so the sequence
$$
(a_0,b_0),(a_1,b_1),\ldots, (a_{s-2},b_{s-2}),(a_0,a_{s-2})
$$
becomes an $s$-chain. By the induction hypothesis, the set $\cup_{\ell=0}^{s-2} \{a_{\ell}\}\cup \{b_{\ell}\} \cup \{a_0\}\cup \{a_{s-2}\}=\cup_{\ell=0}^{s-2} \{a_{\ell}\}\cup \{b_{\ell}\}$ has at most $s+1$ elements. The set $S$ of the $(s+1)$-chain has in addition at most one more new element which implies that $\#S \leq (s+1)+1$.
This proves the claim.
\end{proof}

\subsection{An application of the Combinatorial Lemma} In what follows, let $\gamma$ be a degenerate $k$-periodic orbit of $f$ and let
$x= \min \gamma$. For the next results, we assume that
\begin{equation}\label{star}
\gamma\subset [0,1)\setminus W_{m+1}.
\end{equation}
The hypothesis (\ref{star}) will be removed in Lemma \ref{xxx}.

\begin{lemma}\label{ttoo} $\gamma\subset \cup_{j=1}^{m+1}\partial W_j$.
\end{lemma}
\begin{proof} By (\ref{01}), it is enough
to prove that $\gamma\subset [0,1)\setminus W_j$ for all $j\in\{1,\ldots,m+1\}$. Firstly we consider $j\in \{1,\ldots,m\}$. In this case, there exists a regular periodic orbit $\gamma_j$ such that $\omega(y)=\gamma_j$ for all $y\in W_j$. In particular, if $\gamma\cap W_j\neq\emptyset$ and $y\in\gamma\cap W_j$ then $\gamma=\omega(x)=\omega(y)=\gamma_j$, which is a contradiction, because $\gamma$ is a degenerate periodic orbit. Thus, $\gamma\subset [0,1)\setminus W_j$ for every $j\in\{1,\ldots,m\}$. By (\ref{star}), $\gamma\subset [0,1)\setminus W_{m+1}$. Hence, $\gamma\subset [0,1)\setminus\cup_{j=1}^{m+1} W_j$. By (\ref{01}), $[0,1)\setminus\cup_{j=1}^{m+1} W_j=\cup_{j=1}^{m+1} \partial W_j$.
\end{proof}

\begin{lemma}\label{lemf2} There exist integers $s\ge 1$ and $0\le {\ell}_0<{\ell}_1<\ldots <{\ell}_{s-1}\le k-1$ such that 
$\gamma\cap \{x_0,\ldots,x_{n-1}\}=\{f^{{\ell}_0}(x),f^{{\ell}_1}(x),\ldots, f^{{\ell}_{s-1}}(x)\}$. 
\end{lemma}
\begin{proof} It follows immediately from Proposition \ref{intreg}.
\end{proof}

Because $[0,1)=\cup_{j=1}^n I_j$, 
for each $\ell\in\{0,\ldots,k-1\}$, there exists a unique $j(\ell)\in\{1,\ldots,n\}$ such that
$f^\ell(x)\in I_{j(\ell)}$. 

\begin{lemma}\label{ababab} Let $\{\ell_0,\ell_1\ldots,\ell_{s-1}\}$ be as in Lemma \ref{lemf2}.
For each $\ell\in\{\ell_0,\ell_1,\ldots,\ell_{s-1}\}$, there exists a uniquely defined ordered pair \mbox{$(a_{\ell},b_{\ell})\in \{1,\ldots,m+1\}\times
\{1,\ldots,m+1\}$}\label{lemf1} satisfying the  following conditions:
\begin{itemize}
\item [(a)] $f^{\ell}(x)\in {\rm int}\,(\overline{W_{a_{\ell}}}\cup \overline{W_{b_{\ell}}})$ if $f^{\ell}(x)\neq 0$;
\item [(b)] $f^{\ell}(x)\in\overline{W}_{a_{\ell}}$ and $a_{\ell}=b_{\ell}$  if $f^{\ell}(x)=0$;
\item [(c)] $I_{j({\ell})}\cap (f^{\ell}(x)-\epsilon,f^{\ell}(x)+\epsilon)\subset\overline W_{a_{\ell}}$ for $\epsilon>0$ small enough.
\end{itemize}
\end{lemma}
\begin{proof} Let $\ell\in\{\ell_0,\ell_1,\ldots,\ell_{s-1}\}$, thus there exists a unique integer $j(\ell)\in\{1,\ldots,n\}$ such that
\mbox{$f^\ell(x)\in I_{j(\ell)}\cap \partial I_{j(\ell)}$}. By Lemma \ref{ttoo}, $\gamma\subset\cup_{j=1}^{m+1}\partial W_j$. 
By Theorem \ref{sm} and by (\ref{01}),
there exists a unique index \mbox{$a_{\ell}\in \{1,\ldots,m+1\}$} such that
$$I_{j({\ell})}\cap (f^{\ell}(x)-\epsilon,f^{\ell}(x)+\epsilon)\subset\overline {W_{a_{\ell}}}$$
for $\epsilon>0$ small enough. If  $f(x)=0$ or if $\overline{W_{a_\ell}}$ contains the whole interval
$(f(x)-\epsilon,f(x)+\epsilon)$, we set $b_{\ell}=a_{\ell}$.

Otherwise, there exists a unique index $b_{\ell}\in \{1,\ldots,m+1\}$, $b_{\ell}\neq a_{\ell}$, such that
$$(f^{\ell}(x)-\epsilon,f^{\ell}(x)+\epsilon)\cap\overline W_{b_{\ell}}\neq\emptyset$$
for all $\epsilon>0$ small enough. We have proved there exists a unique pair of indices 
$(a_{\ell},b_{\ell})\in \{1,\ldots,m+1\}\times \{1,\ldots,m+1\}\ $ which satisfies (a), (b) and (c).
\end{proof}

  \begin{lemma}\label{les} Let $(a_{\ell},b_{\ell})$, and
  $0\le {\ell}_0<{\ell}_1<\ldots <{\ell}_{s-1}\le k-1$ be as in Lemmas \ref{lemf2} and \ref{lemf1}. The following holds:
    \begin{itemize}
\item [(a)] $A_0=(a_{{\ell}_0},b_{{\ell}_0})$, $A_1=(a_{{\ell}_1},b_{{\ell}_1})$, $\ldots$, $A_{s-1}=(a_{{\ell}_{s-1}},b_{{\ell}_{s-1}})$ is an $s$-chain;
\item [(b)]  $\#S\le s$;
\item [(c)] If $0\in \gamma$ then $\#S\le s-1$.
\end{itemize}
\end{lemma} 
\begin{proof} 
Let $r\in\{0,\ldots,s-1\}$. For convenience we set
$\ell_s=\ell_0+k$, $a_{\ell_s}=a_{\ell_0}$ and $b_{\ell_s}=b_{\ell_0}$. Notice that, because $x$ is $k$-periodic, $f^{\ell_s}(x)=f^{\ell_0}(x)$. 

By Lemma \ref{arefi} and by the continuity of $f$ on $f^{\ell}(x)$ for all $\ell\in \{0,\ldots,k-1\}\setminus \{\ell_0,\ldots,\ell_{s-1}\}$, we have that $f^{\ell_{r+1}}(x)\in \overline{W_{a_{\ell_r}}}$ for all $r\in\{0,\ldots,s-1\}$. By the unicity in the definition of $(a_{\ell_{r+1}},b_{\ell_{r+1}})$ (see Lemma \ref{ababab}), we have that  $a_{\ell_{r+1}}=a_{\ell_r}$ or $b_{\ell_{r+1}}=a_{\ell_r}$.
Thus, $A_0$, $A_1$, \ldots,
$A_{s-1}$ is an $s$-chain. By Lemma \ref{comb}, $\#S\le s+1$, where $S$ is the set of coordinates of the chain. Moreover, if $\#S=s+1$ then
\begin{equation}\label{migual}
a_{\ell_0}=a_{\ell_1}=\cdots=a_{\ell_{s-1}}.
\end{equation}
By the equation (\ref{migual}), there exists $\epsilon>0$ and an interval $U$ containing $f^{\ell_0}(x)$ such that 
$f^{\ell}(U)$ is an interval containing $f^{\ell+\ell_0}(x)$ for all $\ell\in\{0,\ldots,k\}$.
Now there are two possibilities: either  (i) $f^k(U)\subset U$ or (ii) $f^k(U)\cap U=\{f^{\ell_0}(x)\}$. The case (i) implies that $f^{\ell_0}(x)$ is a regular periodic point, which contradicts the assumption that $\gamma=O_f(x)$ is a degenerate periodic orbit.  In the case (ii) we have that $a_{\ell_0}=b_{\ell_0}$, which together with the second statement of Lemma \ref{comb} imply that $\#S\le s$. The items (a) and (b) of the assertion of the lemma are proved.

Now suppose that $0\in\gamma$. By item (c) of Lemma \ref{ababab}, $a_{i_0}=b_{i_0}$. Consequently, 
$$A_1=(a_{{\ell}_1},b_{{\ell}_1}), \ldots, A_{s-1}=(a_{{\ell}_{s-1}},b_{{\ell}_{s-1}})$$ 
is an ($s-1$)-chain. By the above, $\cup_{r=1}^{s-1} \{a_r\}\cup \{b_r\}$ has at most $s-1$ elements. Moreover, as $a_{i_0}\in \{a_{i_1},b_{i_1}\}$, we have that $S=\cup_{r=0}^{s-1} \{a_{\ell_r}\}\cup \{b_{\ell_r}\}=\cup_{r=1}^{s-1} \{a_{\ell_r}\}\cup \{b_{\ell_r}\}$ and so $\#S\le s-1$, which proves the item (c).
\end{proof}

\begin{lemma}\label{ssss} The cardinality of the set $ \{j\in\{1,...,m+1\}: \inf W_j \in \gamma\}$ is at most $s-1$. 
\end{lemma}
\begin{proof} We claim that 
\begin{equation}\label{claim3}
\# \{j\in\{1,...,m+1\}: \inf W_j \in \gamma\} =
\# \{i\in S: \inf W_{i} \in \gamma\},
\end{equation}
where $S=\cup_{r=0}^{s-1} \{a_{\ell_r}\}\cup \{b_{\ell_r}\}$.

Suppose that $\inf W_j\in\gamma$, thus there exist $r\in\{0,1,\ldots,s-1\}$
and $\ell_{r}<\ell\le \ell_{r+1}$ such that $f^{\ell}(x)=\inf W_j$, where for convenience we set
$\ell_s=\ell_0+k$, $a_{\ell_s}=a_{\ell_0}$ and $b_{\ell_s}=b_{\ell_0}$. Notice that, because the point $x=\min\gamma$ is $k$-periodic, $f^{\ell_s}(x)=f^{\ell_0}(x)$. By Lemma \ref{arefi} and the continuity of $f$ at $f^{\ell}(x)$ for every $\ell\in \{0,\ldots,k-1\}\setminus \{\ell_0,\ldots,\ell_{s-1}\}$, we have that  $f^{\ell_{r+1}}(x)\in \overline{W_{j}}$ for every $r\in\{0,\ldots,s-1\}$. By the definition of $(a_{\ell_{r+1}},b_{\ell_{r+1}})$ (see Lemma \ref{ababab}), we have that $a_{\ell_{r+1}}=j$ or $b_{\ell_{r+1}}=j$. Hence,
$$\inf W_j\in \{\inf W_{a_{\ell_{r+1}}}, \inf W_{b_{\ell_{r+1}}} \}\subset \{i\in S: \inf W_{i} \in \gamma\},$$
which proves (\ref{claim3}).

By (\ref{claim3}), it suffices to prove that $\# \{i\in S: \inf W_{i} \in \gamma\}\le s-1$. It follows from the \mbox{item (c)} of Lemma \ref{les}, that if $0\in\gamma$ then
$$\# \{i\in S: \inf W_{i} \in \gamma\}\le \#S\le s-1.$$ 
Otherwise, $0\not\in\gamma$ and $f^{\ell_0}(x)>0$. Moreover, there exists $i(x)\in S$ such that $x\in\overline{W_{i(x)}}$ and $\inf W_{i(x)}<x$. This together with the item (b) of  Lemma \ref{les} yield
$$\# \{i\in S: \inf W_{i} \in \gamma\}\le \#S-1\le s-1.$$ 
 \end{proof}
 
 Let $\beta:\mathcal{W}\to \{x_0,x_1,\ldots,x_{n-1}\}$ be the map defined in (\ref{beta}), where $\mathcal{W}=\{W_1,\ldots,W_{m+1}\}$ if $W_{m+1}\neq\emptyset$, otherwise
 $\mathcal{W}=\{W_1,\ldots,W_m\}$.
 
  Let $\gamma_1$, \ldots, $\gamma_m$ and $\tilde\gamma_{1},\ldots, \tilde{\gamma_{d}}$ be, respectively, collections of regular and degenerate periodic orbits of $f$. 

\begin{lemma}\label{xxx} The image of the map $\beta$ contains no more than $n-d$ elements.
\end{lemma}
\begin{proof} Let ${\ell}\in\{1,\ldots,d\}$. We claim that 
\begin{equation}\label{edd}
\# \Big(\tilde{\gamma_{\ell}}\cap \textrm{image}\,(\beta)\Big)\le \# \Big(\tilde{\gamma_{\ell}}\cap\{x_0,\ldots,x_{n-1}\}\Big)-1.
\end{equation}
We split the proof of the claim into three cases.
\begin{itemize} 
\item [Case (i)]  $W_{m+1}\neq\emptyset$ and $\tilde{\gamma_{\ell}}\subset [0,1)\setminus W_{m+1}$.

Let $x_i\in  \tilde{\gamma_{\ell}}\cap \textrm{image}\,(\beta)$ and let $j\in\{1,\ldots,m+1\}$ be such that
$x_i=\beta(W_j)$. By the definition of $x_i$ and $\beta$, we have that $x_i\in \tilde{\gamma_{\ell}}\cap
O_f(\inf W_j)$, thus $O_f(\inf W_j)=\tilde{\gamma_\ell}$. In particular, $\inf W_j\in\tilde{\gamma_\ell}$.
This together with the fact that 
 ${\rm image}\,(\beta)\subset \{x_0,\ldots,x_{n-1}\}$, Lemma \ref{ssss} and the injectivity of $\beta$ yield
\begin{eqnarray*}
\#\Big( \tilde{\gamma_{\ell}}\cap \textrm{image}\,(\beta)\Big) &=& \# \Big( \{x_0,\ldots,x_{n-1}\}\cap\tilde{\gamma_{\ell}}\cap {\rm image}\,(\beta)\Big)  \\ &\le& \# \Big( \{j\in\{1,...,m+1\}: \inf W_j \in \tilde{\gamma_{\ell}}\}\Big)  \le \#\Big( \tilde{\gamma_{\ell}}\cap \{x_0,\ldots,x_{n-1}\}\Big)-1,
\end{eqnarray*}
which proves the claim in Case (i).

\item [Case (ii)] $W_{m+1}\neq\emptyset$ and $\tilde{\gamma_{\ell}}\cap W_{m+1}\neq\emptyset$.

In this case, by \mbox{Lemma \ref{arefi}}, we have that $\tilde{\gamma_{\ell}}\subset \overline{W_{m+1}}$.
Moreover, as \mbox{$\tilde{\gamma_{\ell}}\cap W_{m+1}\neq\emptyset$}, we cannot have
$\tilde{\gamma_\ell}\subset\partial W_{m+1}$. Hence, there are two possibilities: either (a)
$\tilde{\gamma_\ell}\subset W_{m+1}$ or (b) $\tilde{\gamma_{\ell}}\cap W_{m+1}\cap\partial W_{m+1}\neq\emptyset$. In the case (a), because $W_{m+1}$ is open and  ${\rm image}\,(\beta)\subset \cup_{j=1}^{m+1} \partial W_j$, we have that
$\tilde{\gamma_\ell}\cap {\rm image}\,(\beta)=\emptyset$, and thus (\ref{edd}) holds. In case (b), by Lemma \ref{inside}, $\tilde{\gamma_{\ell}}\cap \{x_0,\ldots,x_{n-1}\}\cap \partial W_{m+1}\neq\emptyset$. Moreover, by the hypothesis of case (b), there exists $z\in W_{m+1}$ and $x_i\in \tilde{\gamma_{\ell}}\cap \{x_0,\ldots,x_{n-1}\}\cap \partial W_{m+1}$ such that $f(z)=x_i$. If $z\in \{x_0,\ldots,x_{n-1}\}$ then, by proceeding as above, we can see that 
$z\not\in {\rm image}\,(\beta)$ and so (\ref{edd}) holds. Otherwise, $f$ is continuous on a neighborhood of $z$ and so
$x_i\in {\rm int}\,(\overline{W_{m+1}})$. In this case, by the definition of $\beta$, $x_i\not\in {\rm image}\,(\beta)$, hence (\ref{edd}) holds. This proves the claim in Case (ii).

\item [ Case (iii)] $W_{m+1}=\emptyset$. 

The proof of the Case (i) holds word-by-word for the Case (iii), provided we replace 
\mbox{$\{1,\ldots,m+1\}$}  by $\{1,\ldots,m\}$  in that proof.
\end{itemize}

By the claim, for each $\ell\in\{1,\ldots,d\}$, there exists
$\tilde{x_{\ell}}\in \tilde{\gamma_{\ell}}\cap\{x_0,\ldots,x_{n-1}\}$ such that \mbox{$\tilde{x_{\ell}}\not\in \textrm{image}\,(\beta)$}. Therefore,
$$
\textrm{image}\,(\beta)\subset \{x_0,\ldots,x_{n-1}\}\setminus \{\tilde{x_1},\ldots,\tilde{x_{d}}\}.
$$
In this way, $\# \textrm{image}\,(\beta) \le n-d$.
\end{proof}

 By Lemma \ref{xyk}, $f$ has at most $n$ regular periodic orbits, thus $m\le n$.
By the \mbox{Proposition \ref{intreg}}, every degenerate periodic orbit of $f$ contains a discontinuity, and so $d\le n$. 
Therefore, a corollary of these two results is that the number of periodic orbits of $f$ is bounded by $2n$, that is, $m+n\le 2n$. By using  Lemma \ref{xxx}, we provide now a stronger version of Lemma \ref{xyk}. 
 
 \begin{lemma}\label{xykmod} The following statements are true:
 \begin{itemize}
 \item [(a)] if $W_{m+1}\neq\emptyset$ then $m+d\le n-1$;
 \item [(b)] if $W_{m+1}=\emptyset$ then $m+d\le n$.
 \end{itemize}
  \end{lemma}
  
  \begin{proof} By  Lemma \ref{xxx}, the image of the injective map $\beta:\mathcal{W}\to \{x_0,\ldots,x_{n-1}\}$ has at most $n-d$ elements. In case (a), $\mathcal{W}=\{W_1,\ldots,W_{m+1}\}$ and so
  $m+1\le n-d$, that is to say, $m+d\le n-1$. In case (b), $\mathcal{W}=\{W_1,\ldots,W_m\}$ and
  $m\le n-d$, that is, $m+d\le n$.
  \end{proof}
  
  \begin{proof}[Proof of Theorem \ref{main}] By items (a) and (b) of Lemma \ref{xykmod}, $f$ has at most $n$ periodic orbits. Moreover, by item (a) of Lemma \ref{xykmod}, if $f$ has $n$ periodic orbits, then $W_{m+1}=\emptyset$. In this case, $\bigcup_{i=1}^m\overline{W_i}=[0,1]$. For every $x\in W_i$, we have that $\omega(x)$ is the periodic orbit $\gamma_i$. Now if $x\in\partial W_i$, then either $O_f(x)\cap W_i\neq\emptyset$ (and so $\omega(x)=\gamma_i)$ or $O_f(x)$ is contained in the finite set $\bigcup_{i=1}^n \partial W_i$ (see Theorem \ref{sm}), and thus $O_f(x)$ is periodic.
    \end{proof}
    
    \section{Proof of Theorem \ref{main2}}

  \begin{proof}[Proof of Theorem \ref{main2}]
    
   By item (c) of Lemma \ref{3itens}, $\{f^{\ell}(F_j)\mid \ell\ge 0, j\in\{1,\ldots, r\}\}$ is a denumerable
    family of pairwise disjoint open intervals whose union $G=\bigcup_{\ell\ge 0} \bigcup_{j=1}^r f^{\ell}(F_j)$ has Lebesgue measure one. Moreover, the subintervals of $G$ generate the Borel $\sigma$-algebra in $[0,1)$.
    Let $K\subset G$ be an interval, then there exist
    $\ell\ge 0$, $1\le j\le r$, and a subinterval $J$ of $F_j$ such that $K=f^{\ell}(J)$. We set
    \begin{equation}\nu(K)=\label{medidaemG} \nu(f^{\ell}(J))=\dfrac{1}{2^{(\ell+1)} r}\dfrac{\vert J\vert}{\vert F_j\vert},\quad \textrm{thus} \quad \nu(f(K))=\dfrac12 \nu (K).
    \end{equation}
    The set function $K\mapsto \nu(K)$ can be extended to a non-atomic Borel probability measure positive on open intervals, as
    $$\nu (G)=\sum_{\ell\ge 0}\sum_{j=1}^r \dfrac{1}{2^{(\ell+1)} r}=\sum_{j=1}^r\dfrac1r=1.$$
In this way, the map $h:[0,1)\to [0,\infty)$ defined by
$$
h(x)=
\begin{cases}
0, & \textrm{if}  \quad x=0 \\ 
\nu((0,x))  & \textrm{if} \quad  0<x<1
\end{cases}
$$
is continuous and strictly increasing. Moreover,
$
h(1)=\nu ((0,1))=\nu (G)=1.$ Therefore, $h:[0,1)\to [0,1)$ is a homeomorphism. Let $\hat{f}:[0,1)\to [0,1)$ be the map defined by $\hat{f}=h\circ f\circ h^{-1}$. We have that $\hat{f}$ is continuous on $[0,1)\setminus \{h(x_1),\ldots, h(x_{n-1})\}$
and its continuity intervals are $\hat{I}_i=h(I_i)$, $1\le i\le n$.
 
    Let  $B\subset [0,1)$ be an interval. Being Lispchitz, $f$ takes $\nu$-null measure set onto $\nu$-null measure set, thus $\nu(f(B))=\nu (f(B\cap G))$. Now it follows from (\ref{medidaemG}) that
    
  \begin{equation}\label{2eq}
     \nu(f(B))=\dfrac12\nu(B),\quad \textrm{for every interval}\,\,B\subset [0,1).
    \end{equation}
    

    Let $(u,v)\subset h(I_i)$ be an interval. If $f\vert _{I_i}$ is increasing then
    \begin{equation}\label{timp}
    \big ( f(h^{-1}(u)),f(h^{-1}(v))\big)=f\big (h^{-1}(u),h^{-1}(v)\big).
    \end{equation}
    By (\ref{2eq}) and (\ref{timp}),
    \begin{eqnarray*}
    \hat{f}(v)-\hat{f}(u)&=&h\big(f(h^ {-1}(v))\big)-h\big(f(h^{-1}(u))\big)=\nu\Big(\big(0,f(h^{-1}(v))\big)\Big)-\nu\Big(\big( 0,f(h^{-1}(u))\big)\Big)\\
    &=&\nu \Big(\big(f(h^{-1}(u)), f(h^{-1}(v))\big) \Big)=\nu\Big(f\big(h^{-1}(u),h^{-1}(v)\big)\Big)\\&=&\dfrac12\nu \Big(h^{-1}(u),h^{-1}(v)\Big)=\dfrac12\Big[ \nu \Big( \big(0,h^{-1}(v)\big)\Big)-\nu \Big( \big(0,h^{-1}(u)\big)\Big)\Big]\\&=&
    \dfrac12 \big[h(h^{-1}(v))-h(h^{-1}(u))\big]=\dfrac12(v-u).
    \end{eqnarray*}
    Otherwise, $f\vert_{I_i}$ is decreasing and
    $$\hat{f}(v)-\hat{f}(u)=-\dfrac12 (v-u).
    $$
    We have proved that $\hat{f}\vert_{\hat{I}_i}$ is linear for every $i\in\{1,\ldots,n\}$.
       
\end{proof}

\appendix
 \section{Piecewise increasing piecewise contractions of $n$ intervals}
 
The aim of this section is to show that the proof of Theorem 1.1 is much simpler if the piecewise contraction is also piecewise increasing (see Theorem A.1). Nevertheless, such simple proof fails for general piecewise contractions. For completeness sake, the proof is presented below.

\begin{theorem} \label{easy} If $f$ is a piecewise contraction of $n$ intervals such that $I_i=[x_{i-1},x_i)$ and $f\vert_{I_i}$ is increasing for every $i\in\{1,\ldots,n\}$, then
$f$ has at most $n$ periodic orbits. Moreover, every periodic orbit of $f$ is regular.
\end{theorem}  

\begin{proof} Let $p\in [0,1)$ be a $k$-periodic point of $f$ and let $\gamma$ be its orbit. As $f$ is injective, the set
$$
\bigcup_{\ell=0}^{k-1}f^{-\ell}\big(\{x_1,\ldots,x_{n-1}\}\big)
$$
has at most $k(n-1)$ points, thus the minimum 
$$
\epsilon(p)=\min \left\{\epsilon'>p \,\Big |\, \epsilon' \in \{1\} \cup  \bigcup_{\ell=0}^{k-1}f^{-\ell}\big(\{x_1,\ldots,x_{n-1}\}\big) \right\}
$$ 
is well defined. Moreover,  $ \epsilon(p)=1$ or there exist $0\leq \ell \leq k-1$ and $1\leq i \leq n-1$ such that $ \epsilon(p)=f^{-\ell}(x_i)$, thus in this case   $f^{\ell} (\epsilon(p))=x_i$. 

As $f$ is uniformly continuous on $[x_{n-1},1)$, for convenience we denote 
$$
\displaystyle f^{m+1}(1)=f^m(\lim_{x \rightarrow 1-}f(x)), \; \mbox{ for every } \; m\geq 0.
$$
We define the interval   $J_p=[p,\epsilon(p))$ and claim that for every $0\le\ell\le k$,
\begin{align}
&f^\ell(J_p)=\big [f^\ell(p),f^\ell(\epsilon(p)) \big) \label{eqqq1}\\
&\bigcap_{m\ge 0} f^{mk}\big( f^\ell(J_p)\big)=\{f^{\ell}(p)\} \label{eqqq2}.
\end{align}
By the definition of $\epsilon(p)$, we have that $(p,\epsilon(p))\cap \{x_1,\ldots,x_{n-1}\}=\emptyset$, thus
$$J_p=[p,\epsilon(p))\subset [x_{i-1},x_i)=I_i\,\,\,\textrm{for some}\,\,1\le i\le n.$$

As $f\vert_{I_i}$ is continuous and increasing, we have that $f(J_p)=\big[f(p),f(\epsilon(p))\big)$. By recurrence, now assume that there exists $0\leq \ell \leq k-1$ such that $f^\ell(J_p)=\big[f^{\ell}(p),f^\ell(\epsilon(p)) \big)$. By the definition of $\epsilon(p)$,  we have that $\big(f^{\ell}(p),f^\ell(\epsilon(p) )\big)\cap\{x_1,\ldots,x_{n-1}\}=\emptyset$, thus
$$f^{\ell} (J_p)=\big[f^\ell(p),f^\ell(\epsilon(p))\big)\subset [x_{i-1},x_i)=I_i\,\,\,\textrm{for some}\,\,1\le i\le n.$$
As $f\vert_{I_i}$ is continuous and increasing, we have that $f^{\ell+1}(J_p)=\big [f^{\ell+1}(p),f^{\ell+1}(\epsilon(p)) \big)$ and (\ref{eqqq1}) follows by induction.

As $f^k(p)=p$ and $f^k$ is $\kappa^k-$Lipschitz on $J_p$,  we have, for every $x\in J_p$,
$$
0\leq f^k(x)-p\leq \kappa^k( x-p)<x-p.
$$
Therefore, $f^{k}(J_p)\subset J_p$ and $\bigcap_{m\ge 0}f^{mk}(J_p)=\{p\}$. In the same way,
$f^k\big({f^\ell}(J_p)\big)\subset f^\ell(J_p)$ and thus $\bigcap_{m\ge 0}f^{mk}(f^\ell(J_p))=\{f^\ell(p)\}$ for all $\ell\ge 0$. This proves (\ref{eqqq2}).

By the definition of $\epsilon(p)$, there exist $0\le \ell\le k-1$ and $1\le i\le n$ such that
\mbox{$f^\ell(J_p)=[f^\ell(p),x_i)$}. In this way, we may define a map $\alpha:\gamma\mapsto x_i$ that assigns to each periodic orbit
$\gamma$ of $f$, a discontinuity $x_i=\alpha(\gamma)$. 

We claim that the map $\alpha$ is injective. Let $\gamma_1,\gamma_2$ be periodic orbits of $f$ and let \mbox{$p_j\in\gamma_j$} be a \mbox{$k_j$-periodic} point for every $j\in\{1,2\}$. Assume that \mbox{$\alpha(\gamma_1)=\alpha(\gamma_2)=x_i$} and let
\mbox{$1\le \ell_1,\ell_2\le k-1$} be such that $f^{\ell_j}(J_{p_j})=\big[f^{\ell_j}(p_j),x_i \big)$ for every $j\in\{1,2\}$. We may assume that $f^{\ell_1}(p_1)<f^{\ell_2}(p_2)<x_i$. Hence, the $k_2$-periodic point $q_2=f^{\ell_2}(p_2)$ belongs to $f^{\ell_1}(J_{p_1})$ and thus $q_2\in\bigcap_{m\ge 0} f^{mk_1k_2}\big( f^{\ell_1}(J_{p_1})\big)$.
 On the other hand, by (\ref{eqqq2}), $\bigcap_{m\ge 0} f^{mk_1k_2}\big( f^{\ell_1}(J_{p_1})\big)=\{f^{\ell_1}(p_1)\}$. Hence, $f^{\ell_1}(p_1)=q_2$ and so $\gamma_1=\gamma_2$.

We have proved that there is an injective map from the set of periodic orbits of $f$ to the set $\{x_1,\ldots,x_n\}$. Therefore $f$ has at most $n$ periodic orbits.
\end{proof}

In Example 3, the map $h_2$ shows that the proof of Theorem \ref{easy} only holds if the continuity interval $I_i$
is equal to $[x_{i-1},x_i)$, otherwise degenerate periodic orbits may appear. Furthermore, such proof fails if $f\vert_{[x_{i-1},x_i)}$ is decreasing for some $1\le i \le n$.

The argument of the proof of Theorem \ref{easy}  does not suffice to prove \mbox{Theorem \ref{main}}. We recall that the main idea in our proof was to construct an injective map  assigning to each (attractive) periodic orbit $\gamma$ a point of
the set $\{x_1,\ldots,x_n\}$ belonging to the boundary of the trapping region ${\Omega(\gamma)}$. The best we can reach with such argument is the bound $3n$ for the number of periodic orbits of all piecewise contractions.

\end{document}